\documentclass[12pt,reqno]{amsart}

\usepackage{amssymb}
\usepackage[all]{xy}

\newtheorem{theorem}{Theorem}[section]
\newtheorem{corollary}[theorem]{Corollary}

\numberwithin{equation}{section}

\begin{document}

\title[Poisson structure on character varieties, II]{Poisson structure on character varieties, II}

\author[I. Biswas]{Indranil Biswas}

\address{Department of Mathematics, Shiv Nadar University, NH91, Tehsil
Dadri, Greater Noida, Uttar Pradesh 201314, India}

\email{indranil.biswas@snu.edu.in, indranil29@gmail.com}

\author[L. C. Jeffrey]{Lisa C. Jeffrey}

\address{Department of Mathematics,
University of Toronto, Toronto, Ontario, Canada}

\email{jeffrey@math.toronto.edu}



\date{}

\begin{abstract}
Let $G$ be a complex reductive group and $D\, \subset\, X$ a finite subset of a compact Riemann surface $X$.
It was shown in \cite{BJ} that the moduli space of $G$--characters of $\pi_1(X\setminus D)$ has a
natural Poisson structure. We show that the moduli space of logarithmic $G$--connections on $X$ singular over $D$
has a Poisson structure. It is proved that the monodromy map from the moduli space of logarithmic $G$--connections
to the moduli space of $G$--characters is Poisson structure preserving.
\end{abstract}

\maketitle

\section{Introduction}\label{sec1}

Let $X$ be a $C^\infty$ compact oriented surface and $D\, \subset\, X$ a finite subset.
Let $G$ be a reductive Lie group. In \cite{BJ} it was shown that the $G$--character variety
${\mathcal R}(G)$ for $\pi_1(X\setminus D)$ has a natural Poisson structure. When $G$ is
an algebraic group, the $G$--character variety ${\mathcal R}(G)$ has an algebraic structure,
because the group $\pi_1(X\setminus D)$ is finitely presented. In that case, the
Poisson structure on ${\mathcal R}(G)$ is in fact algebraic.

Now take $X$ to be a compact connected Riemann surface; as before, $D\, \subset\, X$ is a finite
subset. Let $G$ be a complex reductive algebraic group. Denote by ${\mathcal M}_X(G)$ the moduli
space of irreducible logarithmic $G$--connections on $X$ singular over $D$. We show that there
is an algebraic homomorphism
$$
\Psi\,\ :\, \ T^*{\mathcal M}_X(G)\,\ \longrightarrow\,\ T{\mathcal M}_X(G)
$$
(see \eqref{e10}).

Let
$$
F\,\ :\,\ {\mathcal M}_X(G)\, \ \longrightarrow\,\ {\mathcal R}(G)
$$
be the map that sends any logarithmic connection to its monodromy homomorphism. This map $F$
is a biholomorphism, but it is not algebraic \cite{Si1}, \cite{Si2} (the variety ${\mathcal R}(G)$
is affine, while the Krull dimension of the space of regular functions on ${\mathcal M}_X(G)$
is strictly less than the dimension of ${\mathcal M}_X(G)$ \cite{BIKS}, \cite{BRag}).

We prove that $F$ takes the homomorphism $\Psi$ to the homomorphism
$\Phi\, :\, T^*{\mathcal R}(G)\,\ \longrightarrow\,\ T{\mathcal R}(G)$ defining the Poisson structure on
${\mathcal R}(G)$; see Theorem \ref{thm1}.

As a consequence, $\Psi$ is a Poisson structure on ${\mathcal M}_X(G)$; see Corollary \ref{cor1}.

It may be noted that although the map $F$ is not algebraic, it still takes the algebraic homomorphism
$\Psi$ to the algebraic homomorphism $\Phi$.

When $D$ is the empty subset, then $\Phi$ and $\Psi$ are isomorphisms, and they define algebraic
symplectic structures on ${\mathcal R}(G)$ and ${\mathcal M}_X(G)$ respectively. In that case,
Theorem \ref{thm1} was proved earlier in \cite{Bi}.

\section{Moduli space of connections}\label{sec2}

Let $X$ be a compact connected Riemann surface. Its holomorphic tangent and cotangent bundles
will be denoted by $TX$ and $K_X$ respectively.
(Note that $K_X$ is the top exterior power of the holomorphic cotangent bundle,
but since $X$ has complex dimension $1$ its holomorphic cotangent bundle
has complex dimension $1$ so it is the same as $K_X$.)
Fix finitely many points
\begin{equation}\label{e1}
D\ :=\ \{x_1, \, \cdots, \, x_m\}\ \subset\ X
\end{equation}
of $X$. Let
\begin{equation}\label{e2}
X_0\ :=\ X\,\setminus\, \{x_1, \, \cdots, \, x_m\} \ =\ X\,\setminus\, D
\end{equation}
be the $m$-punctured Riemann surface.

For a holomorphic vector bundle $\mathcal V$ on $X$, the
vector bundles ${\mathcal V}\otimes{\mathcal O}_X(D)$ and ${\mathcal V}\otimes{\mathcal O}_X(-D)$ 
will be denoted by ${\mathcal V}(D)$ and ${\mathcal V}(-D)$ respectively.

Let $G$ be a complex reductive affine algebraic group. The Lie algebra of $G$ will
be denoted by $\mathfrak g$. Let
\begin{equation}\label{e3}
p\ :\ E_G\ \longrightarrow\ X
\end{equation}
be a holomorphic principal $G$--bundle on $X$.
Its adjoint vector bundle $\text{ad}(E_G)$ is the holomorphic vector bundle on $X$
associated to $E_G$ for the adjoint action of $G$ on its Lie algebra $\mathfrak g$.
Since $G$ is reductive, there is a nondegenerate $G$--invariant bilinear form
$$
\sigma\ :\ \text{Sym}^2(\mathfrak g)\ \longrightarrow\ {\mathbb C}.
$$
Since $\sigma$ is $G$--invariant, it produces a bilinear pairing
\begin{equation}\label{e3b}
\widehat{\sigma}\ :\ \text{Sym}^2(\text{ad}(E_G))\ \longrightarrow\ {\mathcal O}_X.
\end{equation}
We note that $\widehat{\sigma}$ is fiberwise nondegenerate, because $\sigma$ is
nondegenerate.

Consider the action of $G$ on the holomorphic vector bundle $(TE_G)\otimes
{\mathcal O}_{E_G}(-p^{-1}(D))$, where $D$ and $p$ are as in \eqref{e1} and \eqref{e3}
respectively. The corresponding quotient $${\rm At}(E_G)(-\log D)\, :=\, ((TE_G)\otimes{\mathcal O}_{E_G}
(-p^{-1}(D)))/G$$ is a holomorphic vector bundle on $X$ \cite{De}, \cite{At}. This vector bundle
${\rm At}(E_G)(-\log D)$ fits in the following short exact sequence:

\begin{equation}\label{e4}
0 \, \longrightarrow\, \text{ad}(E_G) \, \longrightarrow\, {\rm At}(E_G)(-\log D)
\, \stackrel{p'}{\longrightarrow}\, TX(-\log D)\, \longrightarrow\, 0.
\end{equation}
Here
\begin{equation} \label{log} TX(-\log D)\ =\ TX(-D).
\end{equation}
Also, $p'$ is given by the differential $dp\, :\, TE_G\, \longrightarrow\, p^*TX$ of the
map $p$ in \eqref{e3}.

A \textit{logarithmic connection} on $E_G$ singular over $D$ (see \eqref{e1}) is a holomorphic
homomorphism
$$
\nabla'\ :\ TX(-\log D)\ \longrightarrow\ {\rm At}(E_G)(-\log D)
$$
such that $p'\circ \nabla'\,=\, \text{Id}_{TX(-\log D)}$, where $p'$ is the projection in \eqref{e4}
\cite{De}, \cite{At}. A logarithmic connection $\nabla'$ is called \textit{reducible} if there is a
proper parabolic subgroup $P\, \subset\, G$, and a holomorphic reduction of the structure group $E_P\,
\subset\, E_G$ of $E_G$ to $P$, such that there is a logarithmic connection $\nabla^P$ with the
property that the logarithmic connection on $E_G$ induced by $\nabla^P$ coincides with $\nabla$. A
logarithmic connection $\nabla'$ is called \textit{irreducible} if it is not reducible.

Let
\begin{equation}\label{em}
{\mathcal M}_X(G)
\end{equation}
denote the moduli space of irreducible logarithmic $G$--connections
on $X$ singular over $D$; see \cite{Ni}, \cite{Si1}, \cite{Si2} for the construction
of this moduli space.

Take a holomorphic principal $G$--bundle $E_G$ on $X$ equipped with a logarithmic connection $\nabla'$
singular over $D$. The logarithmic connection on $\text{ad}(E_G)$ induced by $\nabla'$ will be
denoted by $\nabla$. So $\nabla$ is a holomorphic differential operator of order one
\begin{equation}\label{e5}
\nabla\ :\ \text{ad}(E_G)\ \longrightarrow\ \text{ad}(E_G)\otimes K_X (D)\ :=\
\text{ad}(E_G)\otimes K_X \otimes{\mathcal O}_{X}(D),
\end{equation}
where $K_X$ is the holomorphic cotangent bundle of $X$, that satisfies the Leibniz identity. Let
${\mathcal C}_\bullet$ be the following two-term complex of sheaves on $X$:
\begin{equation}\label{e6}
{\mathcal C}_\bullet\ :\ {\mathcal C}_0\ :=\ \text{ad}(E_G) \ \stackrel{\nabla}{\longrightarrow}\
{\mathcal C}_1\ :=\ \text{ad}(E_G)\otimes K_X(D),
\end{equation}
where $\nabla$ is the differential operator in \eqref{e5}. The space of infinitesimal deformations of
the pair $(E_G,\, \nabla')$ are parametrized by the first hypercohomology ${\mathbb H}^1({\mathcal C}_\bullet)$,
where ${\mathcal C}_\bullet$ is the complex in \eqref{e6} \cite[p.~554, Proposition 3.2(1)]{ABDH}
\cite[p.~1415, Proposition 4.4]{Ch}, \cite{BR}.

The operator $\nabla$ in \eqref{e5} sends the subsheaf $\text{ad}(E_G)\otimes{\mathcal O}_{X}(-D)\,
\subset\, \text{ad}(E_G)$ to the subsheaf $\text{ad}(E_G)\otimes K_X\, \subset\, \text{ad}(E_G)\otimes K_X(D)$.
Indeed, this follows immediately from the Leibniz identity. Consequently, from \eqref{e6}
we have the two-term complex ${\mathcal C}'_\bullet$ of sheaves on $X$
\begin{equation}\label{e7}
{\mathcal C}'_\bullet\ :\ {\mathcal C}'_0\ :=\ \text{ad}(E_G)\otimes{\mathcal O}_{X}(-D) \
\stackrel{\nabla}{\longrightarrow}\ {\mathcal C}'_1\ :=\ \text{ad}(E_G)\otimes K_X;
\end{equation}
the restriction of $\nabla$ to $\text{ad}(E_G)\otimes{\mathcal O}_{X}(-D)\, \subset\,
\text{ad}(E_G)$ is also denoted by $\nabla$.

Using the pairing $\widehat{\sigma}$ in \eqref{e3b} it follows that ${\mathcal C}'_\bullet$ in \eqref{e7}
is the Serre dual complex of ${\mathcal C}_\bullet$ in \eqref{e6}. So Serre duality gives the following:
\begin{equation}\label{e8}
{\mathbb H}^1({\mathcal C}_\bullet)^*\,\ =\,\ {\mathbb H}^1({\mathcal C}'_\bullet).
\end{equation}

Consider the commutative diagram
$$
\begin{matrix}
{\mathcal C}'_\bullet & :\ \,& \text{ad}(E_G)\otimes{\mathcal O}_{X}(-D) & \stackrel{\nabla}{\longrightarrow}
& \text{ad}(E_G)\otimes K_X\\
\,\, \Big\downarrow\phi && \,\,\, \Big\downarrow\phi_0 && \,\,\,\Big\downarrow\phi_1\\
{\mathcal C}_\bullet & :\ \,& \text{ad}(E_G) & \stackrel{\nabla}{\longrightarrow}
& \text{ad}(E_G)\otimes K_X (D)
\end{matrix}
$$
where $\phi_0$ and $\phi_1$ are the inclusion maps. From the commutativity of the above diagram it follows
that $\phi$ induces a homomorphism
\begin{equation}\label{e9}
\phi_*\, \ :\,\ {\mathbb H}^1({\mathcal C}'_\bullet)\,\ \longrightarrow \,\ {\mathbb H}^1({\mathcal C}_\bullet)
\end{equation}

Take any $(E_G,\, \nabla')\, \in\, {\mathcal M}_X(G)$ (see \eqref{em}). Since the
space of infinitesimal deformations of the pair $(E_G,\, \nabla')$ are parametrized by
${\mathbb H}^1({\mathcal C}_\bullet)$, we have
\begin{equation}\label{e9a}
T_{(E_G,\, \nabla')}{\mathcal M}_X(G)\ =\ {\mathbb H}^1({\mathcal C}_\bullet).
\end{equation}
So from \eqref{e8} it follows that
\begin{equation}\label{e9b}
T^*_{(E_G,\, \nabla')}{\mathcal M}_X(G)\ =\ {\mathbb H}^1({\mathcal C}'_\bullet).
\end{equation}
Consequently, the pointwise homomorphism in \eqref{e9} produces a ${\mathcal O}_{{\mathcal M}_X(G)}$--linear
homomorphism
\begin{equation}\label{e10}
\Psi\,\ :\, \ T^*{\mathcal M}_X(G)\,\ \longrightarrow\,\ T{\mathcal M}_X(G).
\end{equation}

\section{Poisson structure on a character variety}

Consider $X_0$ in \eqref{e2}. Fix a base point $z_0\, \in\, X_0$. A homomorphism
$\rho\, :\, \pi_1(X_0,\, z_0)\, \longrightarrow \, G$ is called \textit{irreducible} if
the image $\rho(\pi_1(X_0,\, z_0))$ is not contained in some proper parabolic subgroup of $G$.
Let $\text{Hom}^{\rm irr}(\pi_1(X_0,\, z_0),\, G)$ denote the space of all irreducible
homomorphisms from $\pi_1(X_0,\, z_0)$ to $G$. The conjugation action of $G$ on itself produces
an action of $G$ on $\text{Hom}^{\rm irr}(\pi_1(X_0,\, z_0),\, G)$. The corresponding quotient
\begin{equation}\label{e11}
{\mathcal R}(G)\,\ :=\,\, \text{Hom}^{\rm irr}(\pi_1(X_0,\, z_0),\, G)/G
\end{equation}
is called a character variety.

In \cite{BJ} it was shown that ${\mathcal R}(G)$ has a natural Poisson structure
\begin{equation}\label{e12}
\Phi\,\ :\, \ T^*{\mathcal R}(G)\,\ \longrightarrow\,\ T{\mathcal R}(G)
\end{equation}
(see \cite[(2.11)]{BJ}, \cite[Section 3.2]{BJ}).

Take a holomorphic principal $G$--bundle $E_G$ on $X$ equipped with a logarithmic connection $\nabla'$.
Fix a point $p_0\, \in\, (E_G)_{z_0}$ in the fiber of $E_G$ over $z_0$. Identify $(E_G)_{z_0}$ with
$G$ by sending any $g\, \in\, G$ to $p_0g\, \in\, (E_G)_{z_0}$. Consider the monodromy homomorphism
$$
\rho(\nabla')\ :\ \pi_1(X_0,\, z_0)\ \longrightarrow\ G
$$
for the connection $\nabla'$. Note that $\nabla'$ is irreducible if and only if the monodromy
homomorphism $\rho(\nabla')$ is irreducible. Let
\begin{equation}\label{e13}
F\,\ :\,\ {\mathcal M}_X(G)\, \ \longrightarrow\,\ {\mathcal R}(G)
\end{equation}
be the map that sends any logarithmic connection to its monodromy homomorphism; it is known as the
Riemann--Hilbert correspondence. We note that
$F$ in \eqref{e13} is independent of the choices of $z_0\, \in\, X_0$ and $p_0$.

Since $\pi_1(X_0,\, z_0)$ is a finitely presented group, the algebraic structure of $G$ produces
an algebraic structure on ${\mathcal R}(G)$. As $G$ is a complex affine algebraic group, this
${\mathcal R}(G)$ is a complex affine variety; it is known as the Betti moduli space \cite{Si1},
\cite{Si2}. We recall that ${\mathcal M}_X(G)$ is called the de Rham moduli space \cite{Si1},
\cite{Si2}. The map $F$ in \eqref{e13} is locally a biholomorphism, but it is not algebraic.
The homomorphism $\Psi$ in \eqref{e10} is algebraic, and $\Phi$ in \eqref{e12} is algebraic.

\begin{theorem}\label{thm1}
The map $F$ in \eqref{e13} intertwines the homomorphisms $\Psi$ and $\Phi$ (see \eqref{e10} and
\eqref{e12}), in other words, for any $\alpha
\, :=\, (E_G,\, \nabla')\, \in\, {\mathcal M}_X(G)$,
$$
(dF)_{\alpha} \circ \Psi_\alpha\circ (dF)^*_{F(\alpha)} \ \ =\ \ \Phi_{F(\alpha)}
$$
as homomorphisms from the cotangent space $T^*_{F(\alpha)} {\mathcal R}(G)$ to $T_{F(\alpha)} {\mathcal R}(G)$,
where $(dF)_{\alpha}\, :\, T_\alpha {\mathcal M}_X(G)\, \longrightarrow\, T_{F(\alpha)}
{\mathcal R}(G)$ is the differential of $F$ at $\alpha$, and $(dF)^*_{F(\alpha)}$ is its dual;
the homomorphism $\Psi_\alpha$ (respectively, $\Phi_{F(\alpha)}$) is the restriction of
$\Psi$ (respectively, $\Phi$) to $\alpha$ (respectively, $F(\alpha)$).
\end{theorem}

\begin{proof}
Take any $\alpha \, :=\, (E_G,\, \nabla')\, \in\, {\mathcal M}_X(G)$. Denote
\begin{equation}\label{e14}
\rho\ \, =\ \, F(\alpha)\ \, \in\ \, {\mathcal R}(G).
\end{equation}
We will give the Dolbeault (respectively, de Rham) descriptions of
$T_\alpha {\mathcal M}_X(G)$ and $T^*_\alpha {\mathcal M}_X(G)$ (respectively, $T_{\rho} {\mathcal R}(G)$ 
and $T^*_{\rho} {\mathcal R}(G)$).

For a holomorphic vector bundle $\mathcal V$ on $X$, denote by $\Omega^{p,q}_X({\mathcal V})$ the
$C^\infty$ vector bundle $(\bigwedge^p T^{1,0}X)^*\otimes (\bigwedge^q T^{0,1}X)^*\otimes {\mathcal V}$
on $X$ given by the $(p,\,q)$--forms with values in $\mathcal V$. Note that
$\Omega^{p,0}_X({\mathcal V})$ is a holomorphic vector bundle. Let
\begin{equation}\label{ed}
\overline{\partial}_E\ :\ \Omega^{0,0}_X (\text{ad}(E_G))\ \longrightarrow\
\Omega^{0,1}_X(\text{ad}(E_G))
\end{equation}
be the Dolbeault operator on the holomorphic vector bundle $\text{ad}(E_G)$. Similarly,
$$
\overline{\partial}'_E\ :\ \Omega^{1,0}_X (\text{ad}(E_G)\otimes{\mathcal O}_X(D))\ \longrightarrow\
\Omega^{1,1}_X (\text{ad}(E_G)\otimes{\mathcal O}_X(D))
$$
is the Dolbeault operator on the holomorphic vector bundle $\text{ad}(E_G)\otimes K_X(D)$.
Consider the complex of sheaves ${\mathcal C}_\bullet$ in \eqref{e6}. It has the following
Dolbeault resolution:
\begin{equation}\label{e15}
\begin{matrix}
0&& 0\\
\Big\downarrow && \Big\downarrow\\
\text{ad}(E_G) & \stackrel{\nabla}{\longrightarrow} & \text{ad}(E_G)\otimes K_X(D)\\
\Big\downarrow && \Big\downarrow\\
\Omega^{0,0}_X(\text{ad}(E_G))& \stackrel{\nabla'}{\longrightarrow} & \Omega^{1,0}_X
(\text{ad}(E_G)\otimes{\mathcal O}_X(D))\\
\,\,\,\, \Big\downarrow \overline{\partial}_E && \,\,\,\, \Big\downarrow\overline{\partial}'_E\\
\Omega^{0,1}_X (\text{ad}(E_G))
& \stackrel{\nabla''}{\longrightarrow} & \Omega^{1,1}_X(\text{ad}(E_G)\otimes{\mathcal O}_X(D))\\
\Big\downarrow && \Big\downarrow\\
0&& 0
\end{matrix}
\end{equation}
where $\nabla'$ and $\nabla''$ are given by $\nabla$. We have the complex
\begin{equation} \label{e16}
0\, \longrightarrow\, {\textbf A} \, \xrightarrow{\,\, \overline{\partial}_E
\oplus \nabla'\,\,}\,{\textbf B}\,
\xrightarrow{\,\,\, \nabla''+ \overline{\partial}'_E\,\,\, }\,\,
{\textbf C} \,\, \longrightarrow\,\, 0,
\end{equation}
where
\begin{equation} \label{adef} {\textbf A}\ :=\ C^\infty(X,\, \text{ad}(E_G)),
\end{equation}
\begin{equation} \label{bdef} 
{\textbf B}\ :=\ C^\infty(X,\, \Omega^{0,1}_X(\text{ad}(E_G)))
\oplus C^\infty(X,\, \text{ad}(E_G)\otimes K_X(D)), \end{equation}
and 
\begin{equation} \label{cdef} 
{\textbf C}\ :=\ C^\infty(X,\, \Omega^{1,1}_X(\text{ad}(E_G)(D))).
\end{equation}
Since \eqref{e15} is a fine resolution of ${\mathcal C}_{\bullet}$, from \eqref{e16} we have
\begin{equation}\label{e17}
{\mathbb H}^1({\mathcal C}_{\bullet})\,=\, \frac{\text{kernel}(\nabla''+
\overline{\partial}'_E)}{(\overline{\partial}_E\oplus \nabla')(C^\infty(X,\, \text{ad}(E_G)))}.
\end{equation}

The complex ${\mathcal C}'_\bullet$ in \eqref{e7} has the following Dolbeault resolution:
\begin{equation}\label{e18}
\begin{matrix}
0&& 0\\
\Big\downarrow && \Big\downarrow\\
\text{ad}(E_G)\otimes {\mathcal O}_X(-D) & \stackrel{\nabla}{\longrightarrow} & \text{ad}(E_G)\otimes K_X\\
\Big\downarrow && \Big\downarrow\\
\Omega^{0,0}_X(\text{ad}(E_G)\otimes {\mathcal O}_X(-D))& \stackrel{\nabla'}{\longrightarrow} & \Omega^{1,0}_X
(\text{ad}(E_G))\\
\,\,\,\, \Big\downarrow \overline{\partial}_E && \,\,\,\, \Big\downarrow\overline{\partial}'_E\\
\Omega^{0,1}_X (\text{ad}(E_G)\otimes {\mathcal O}_X(-D))
& \stackrel{\nabla''}{\longrightarrow} & \Omega^{1,1}_X(\text{ad}(E_G))\\
\Big\downarrow && \Big\downarrow\\
0&& 0
\end{matrix}
\end{equation}
We have the complex
\begin{equation} \label{e19}
0\,\longrightarrow\,{\textbf A}'\,\, \xrightarrow{\,\, \overline{\partial}_E
\oplus \nabla'\,\,}\,{\textbf B}'\,\,
\,\, \xrightarrow{\,\,\, \delta\, :=\, \nabla''+ \overline{\partial}'_E\,\,\, }\,\,
{\textbf C}' \,\, \longrightarrow\,\, 0,
\end{equation}
where
\begin{equation} \label{apdef}
{\textbf A}'\ :=\ C^\infty(X,\,\text{ad}(E_G)(-D)),\end{equation}

\begin{equation}\label{bpdef}
{\textbf B}'\ :=\ C^\infty(X,\, \Omega^{0,1}_X(\text{ad}(E_G)(-D)))\oplus
 C^\infty(X,\, \text{ad}(E_G)\otimes K_X ), \end{equation}
 and
\begin{equation} \label{cpdef}
{\textbf C}'\ :=\ C^\infty(X,\, \Omega^{1,1}_X(\text{ad}(E_G))).
\end{equation}
In \eqref{e19} the notation $\delta$ is introduced in order to distinguish between the map
$\nabla''+ \overline{\partial}'_E$ acting on $C^\infty(X,\, \Omega^{0,1}_X(\text{ad}(E_G)))
\oplus C^\infty(X,\, \text{ad}(E_G)\otimes K_X)$ and on its subspace
$$C^\infty(X,\, \Omega^{0,1}_X(\text{ad}(E_G)(-D)))\oplus C^\infty(X,\, \text{ad}(E_G)\otimes K_X).$$
Since \eqref{e18} is a fine resolution of ${\mathcal C}'_{\bullet}$, from \eqref{e19} we have
\begin{equation}\label{e20}
{\mathbb H}^1({\mathcal C}'_{\bullet})\,=\, \frac{\text{kernel}(\delta)}{(\overline{\partial}_E\oplus
\nabla')(C^\infty(X,\, \text{ad}(E_G)(-D)))}.
\end{equation}

Consider the commutative diagram
$$
\begin{matrix}
0 & \longrightarrow & {\textbf A}' & \xrightarrow{\,\, \overline{\partial}_E\oplus
\nabla'\,\,}& {\textbf B}' & \xrightarrow{\,\, \nabla''+ \overline{\partial}'_E\,\,} & {\textbf C}'
& \longrightarrow & 0\\
&& \Big\downarrow && \Big\downarrow && \Big\downarrow \\
0 & \longrightarrow & {\textbf A} & \xrightarrow{\,\, \overline{\partial}_E\oplus
\nabla'\,\,} & {\textbf B} & \xrightarrow{\,\, \,\delta\,\,\,} & {\textbf C} 
& \longrightarrow & 0
\end{matrix}
$$
(see \eqref{e16} and \eqref{e19}). From this we have a homomorphism
\begin{equation}\label{e21}
\frac{\text{kernel}(\delta)}{(\overline{\partial}_E\oplus
\nabla')(C^\infty(X,\, \text{ad}(E_G)(-D)))}\,\, \longrightarrow\,\,
\frac{\text{kernel}(\nabla''+
\overline{\partial}'_E)}{(\overline{\partial}_E\oplus \nabla')(C^\infty(X,\, \text{ad}(E_G)))}.
\end{equation}
In view of \eqref{e17} and \eqref{e20}, the homomorphism in \eqref{e21} produces a homomorphism
\begin{equation}\label{e22}
{\mathbb H}^1({\mathcal C}'_{\bullet})\,\, \longrightarrow\,\, {\mathbb H}^1({\mathcal C}_{\bullet}).
\end{equation}

The homomorphism in \eqref{e9} clearly coincides with the homomorphism in \eqref{e22}.

Now take $\rho$ in \eqref{e14}. We will recall a description of $T_\rho {\mathcal R}(G)$.
Note that $\rho$ gives a $C^\infty$ principal $G$--bundle $E'_G$ on $X_0$ (see \eqref{e2} for $X_0$)
equipped with a flat connection $\nabla^\rho_1$. Let $\nabla^\rho$ be the flat connection on the
adjoint vector bundle $\text{ad}(E'_G)$ induced by $\nabla^\rho_1$. We note that
$E'_G\,=\, E_G\big\vert_{X_0}$, and hence $\text{ad}(E'_G)\,=\, \text{ad}(E_G)\big\vert_{X_0}$.
Also, the connection $\nabla^\rho$ coincides with the
flat connection $(\nabla+\overline{\partial}_E)\big\vert_{X_0}$, where $\overline{\partial}_E$
is the Dolbeault operator on $\text{ad}(E_G)$ (see \eqref{ed}).

Let $\underline{\text{ad}}(E'_G)$ be the locally constant sheaf on $X_0$ given by the sheaf of flat
sections of the flat connection $\nabla+\overline{\partial}_E$ on $\text{ad}(E_G)\big\vert_{X_0}$. Then
\begin{equation}\label{e9d}
T_\rho {\mathcal R}(G)\,=\, H^1(X_0,\, \underline{\text{ad}}(E'_G)),\ \ \,
T^*_\rho {\mathcal R}(G)\,=\, H^1_c(X_0,\, \underline{\text{ad}}(E'_G)),
\end{equation}
where $H^1_c(X_0,\, \underline{\text{ad}}(E_G))$ denotes compactly supported first cohomology
of $\underline{\text{ad}}(E'_G)$; see \cite{AB}, \cite{Go}, \cite{BJ}. The map
$$
\Phi (\rho)\ :\ T^*_\rho {\mathcal R}(G)\ \longrightarrow\ T_\rho {\mathcal R}(G)
$$
in \eqref{e12} is given by the natural homomorphism $H^1_c(X_0,\, \underline{\text{ad}}(E'_G))\,
\longrightarrow\, H^1(X_0,\, \underline{\text{ad}}(E'_G))$ \cite[(2.9)]{BJ} (in \cite{BJ}
$\Phi (\rho)$ is denoted by $\widetilde{\Phi}_{\rho}$).

We will now describe $H^1(X_0,\, \underline{\text{ad}}(E'_G))$ and $H^1_c(X_0,\, \underline{\text{ad}}(E'_G))$.
Consider the following exact sequence:
\begin{equation}\label{e23}
0\, \longrightarrow\, \underline{\text{ad}}(E'_G) \, \longrightarrow\, \text{ad}(E'_G)
\, \xrightarrow{\,\,\beta_1 := \nabla^\rho\,\,}\, \text{ad}(E'_G)\otimes T^*X_0
\, \xrightarrow{\,\,\beta_2 := \nabla^\rho\,\,}\, \text{ad}(E'_G)\otimes \bigwedge\nolimits^2 T^*X_0
\, \longrightarrow\, 0.
\end{equation}
The exact sequence in \eqref{e23} produces a complex
\begin{equation}\label{e23b}
0\,\, \longrightarrow\,\, C^\infty(X_0,\, \text{ad}(E'_G))
\,\, \xrightarrow{\,\,\,\beta_1 := \nabla^\rho\,\,\,}\,\, 
C^\infty(X_0,\, \text{ad}(E'_G)\otimes T^*X_0)
\end{equation}
$$
\,\, \xrightarrow{\,\,\,\beta_2 := \nabla^\rho\,\,\,}\,\,
C^\infty(X_0,\, \text{ad}(E'_G)\otimes \bigwedge\nolimits^2 T^*X_0)\,\, \longrightarrow\,\, 0.
$$

Since $\nabla^\rho$ is a flat connection, the exact sequence in \eqref{e23} is a fine resolution of the
locally constant sheaf $\underline{\text{ad}}(E'_G)$, and hence
\begin{equation}\label{e24}
H^1(X_0,\, \underline{\text{ad}}(E'_G))\,\ =\,\ \frac{\text{kernel}(\beta_2)}{\text{image}(\beta_1)};
\end{equation}
the homomorphisms $\beta_1$ and $\beta_2$ are as in \eqref{e23b}.

For $j\, \geq\, 0$, let
\begin{equation}\label{j1}
C^\infty_c(X_0,\, \text{ad}(E'_G)\otimes \bigwedge\nolimits^j T^*X_0)\ \hookrightarrow\
C^\infty(X_0,\, \text{ad}(E'_G)\otimes \bigwedge\nolimits^j T^*X_0)
\end{equation}
be the $C^\infty$ compactly supported $j$--forms on $X_0$ with values in $\text{ad}(E'_G)$. From
\eqref{e23b} we have the complex
\begin{equation}\label{e25}
0\,\, \longrightarrow\, C^\infty_c (X_0,\, \text{ad}(E'_G))
\,\, \xrightarrow{\,\,\,\beta^c_1\,\,\, }\,\, 
C^\infty_c (X_0,\, \text{ad}(E'_G)\otimes T^*X_0)
\end{equation}
$$
\xrightarrow{\,\,\,\beta^c_2\,\,\,}\,\,
C^\infty_c (X_0,\, \text{ad}(E'_G)\otimes \bigwedge\nolimits^2 T^*X_0)\,\, \longrightarrow\,\, 0,
$$
where $\beta^c_1$ and $\beta^c_2$ respectively are the restrictions of $\beta_1$ and $\beta_2$
(see \eqref{e23b} for $\beta_1$ and $\beta_2$). Now we have
\begin{equation}\label{e26}
H^1_c(X_0,\, \underline{\text{ad}}(E'_G))\,\ =\,\ \frac{\text{kernel}(\beta^c_2)}{\text{image}(\beta^c_1)};
\end{equation}
the homomorphisms $\beta^c_1$ and $\beta^c_2$ are as in \eqref{e25}.

Consider the homomorphism $\Phi (\rho)\, :\, T^*_\rho {\mathcal R}(G)\, \longrightarrow\,
T_\rho{\mathcal R}(G)$ in \eqref{e12}. Using the isomorphisms in \eqref{e26} and \eqref{e24}, this
$\Phi (\rho)$ is transformed into a homomorphism
\begin{equation}\label{e26b}
\widetilde{\Phi} (\rho)\,\ :\,\ \frac{\text{kernel}(\beta^c_2)}{\text{image}(\beta^c_1)}\,\ \longrightarrow\,\
\frac{\text{kernel}(\beta_2)}{\text{image}(\beta_1)}.
\end{equation}

On the other hand, the inclusion maps in \eqref{j1} for $0\, \leq\, j\, \leq\, 2$ produce a
natural homomorphism
$$
\frac{\text{kernel}(\beta^c_2)}{\text{image}(\beta^c_1)}\,\ \longrightarrow\,\
\frac{\text{kernel}(\beta_2)}{\text{image}(\beta_1)}.
$$
The homomorphism $\widetilde{\Phi} (\rho)$ in \eqref{e26b} coincides with this natural homomorphism.

In view of \eqref{e9a} and the first isomorphism in \eqref{e9d}, the homomorphism
$$
(dF)_{\alpha} \ :\ T_{(E_G,\, \nabla')}{\mathcal M}_X(G)\ \longrightarrow\
T_\rho {\mathcal R}(G)
$$
in the theorem is given by a homomorphism ${\mathbb H}^1({\mathcal C}_\bullet)\, \longrightarrow\,
H^1(X_0,\, \underline{\text{ad}}(E'_G))$. Now using the isomorphisms in \eqref{e17}
and \eqref{e24} we have
\begin{equation}\label{e27}
(dF)_{\alpha} \ :\ \frac{\text{kernel}(\nabla''+
\overline{\partial}'_E)}{(\overline{\partial}_E\oplus \nabla')(C^\infty(X,\, \text{ad}(E_G)))}
\ \longrightarrow\ \frac{\text{kernel}(\beta_2)}{\text{image}(\beta_1)}.
\end{equation}

Take an open subset $U\, \subset\, X_0$. Let $\gamma^{1,0}$ (respectively, $\gamma^{0,1}$)
be a $C^\infty$ section, defined over $U$, of the vector bundle $\Omega^{1,0}(\text{ad}(E_G))$
(respectively, $\Omega^{0,1}(\text{ad}(E_G))$). Then the following three statements are equivalent:
\begin{enumerate}
\item $\nabla^\rho(\gamma^{1,0}+\gamma^{0,1})\ =\ 0$ (see \eqref{e23}).

\item $\overline{\partial}'_E (\gamma^{1,0})\ =\ 0\ =\ \nabla''(\gamma^{0,1})$ (see \eqref{e15}).

\item $(\nabla'' +\overline{\partial}'_E)(\gamma^{1,0}+\gamma^{0,1})\ =\ 0$.
\end{enumerate}

The equivalence between the second statement and the third statement follows from the fact that
$\Omega^{2,0}\,=\,0\,=\, \Omega^{0,2}$ on $X$. The first and the third statements are equivalent,
because
\begin{equation}\label{ei}
\nabla^\rho\ \,=\,\ \nabla'' +\overline{\partial}'_E.
\end{equation}
Recall that 
$\text{ad}(E'_G)\,=\, \text{ad}(E_G)\big\vert_{X_0}$.
Let $\gamma$ be a $C^\infty$ section, defined over $U$, of the vector bundle $\text{ad}(E_G)$.
Then clearly, $(\overline{\partial}_E +\nabla')(\gamma)\,=\, \nabla^\rho (\gamma)$ (see
\eqref{ei}, \eqref{e23} and \eqref{e15}).

Using these it follows that there is a natural homomorphism
\begin{equation}\label{a1}
\frac{\text{kernel}(\nabla''+
\overline{\partial}'_E)}{(\overline{\partial}_E\oplus \nabla')(C^\infty(X,\, \text{ad}(E_G)))}
\ \longrightarrow\ \frac{\text{kernel}(\beta_2)}{\text{image}(\beta_1)}.
\end{equation}
This homomorphism evidently coincides with the homomorphism $(dF)_{\alpha}$ in \eqref{e27}.
This also shows that $(dF)_{\alpha}$ is an isomorphism.

In view of \eqref{e9b} and the second isomorphism \eqref{e9d}, the homomorphism 
$$
(dF)^*_{F(\alpha)} \ :\ T^*_\rho {\mathcal R}(G) \ \longrightarrow\
T^*_{(E_G,\, \nabla')}
$$
in the theorem is given by a homomorphism $H^1_c(X_0,\, \underline{\text{ad}}(E'_G)) \, \longrightarrow\,
{\mathbb H}^1({\mathcal C}'_\bullet)$. Now using the isomorphisms in \eqref{e20}
and \eqref{e26} we have
\begin{equation}\label{e27b}
(dF)^*_{\alpha} \ :\ \frac{\text{kernel}(\beta^c_2)}{\text{image}(\beta^c_1)}\ \longrightarrow\
\frac{\text{kernel}(\delta)}{(\overline{\partial}_E\oplus
\nabla')(C^\infty(X,\, \text{ad}(E_G)(-D)))}.
\end{equation}

Just as the homomorphism in \eqref{a1} is constructed, we can construct a natural homomorphism
$$
\frac{\text{kernel}(\beta^c_2)}{\text{image}(\beta^c_1)}\ \longrightarrow\
\frac{\text{kernel}(\delta)}{(\overline{\partial}_E\oplus
\nabla')(C^\infty(X,\, \text{ad}(E_G)(-D)))}.
$$
The homomorphism $(dF)^*_{\alpha}$ in \eqref{e27b} coincides with it.

All the four homomorphisms in the theorem, namely $(dF)_{\alpha}$, $(dF)^*_{\alpha}$, $\Phi (\rho)$
and $\Psi_\alpha$, have been explicitly described. Now it is straightforward to deduce the
equality of maps in the theorem. Indeed, take any compacted supported $\text{ad}(E_G)$--valued $1$--form
$\omega\, \in\, \text{kernel}(\beta^c_2)$. Let $$[\omega]\ \in\ 
\frac{\text{kernel}(\beta^c_2)}{\text{image}(\beta^c_1)}\ =\ H^1_c(X_0,\, \underline{\text{ad}}(E'_G))
\ =\ T^*_{\rho} {\mathcal R}(G)$$
be the class of $\omega$. Then both the maps $(dF)_{\alpha} \circ \Psi_\alpha\circ (dF)^*_{F(\alpha)}$
and $\Phi_{F(\alpha)}$ in the theorem sends $[\omega]$ to the cohomology class of $\omega$ in
$$
\frac{\text{kernel}(\beta_2)}{\text{image}(\beta_1)}\ =\ H^1(X_0,\, \underline{\text{ad}}(E'_G))
\ =\ T_{\rho} {\mathcal R}(G).$$
This completes the proof.
\end{proof}

\begin{corollary}\label{cor1}
The homomorphism $\Psi$ in \eqref{e10} defines a Poisson structure on ${\mathcal M}_X(G)$.
\end{corollary}

\begin{proof}
The homomorphism $\Phi$ in \eqref{e12} gives a Poisson structure on ${\mathcal R}(G)$ \cite{BJ}.
So from Theorem \ref{thm1} it follows immediately that $\Psi$ defines a
Poisson structure on ${\mathcal M}_X(G)$.
\end{proof}

\section*{Acknowledgements}

The first-named author is partially supported by a J. C. Bose Fellowship (JBR/2023/000003).

The second-named author is partially supported by an NSERC Discovery Grant.

\section*{Data availability}

No data was used or generated in the article.


\end{document}